\documentclass[a4paper, twoside,12pt]{article}
\usepackage[english]{babel}
\usepackage{amsmath,amsthm}
\usepackage{amsfonts}
\usepackage{amssymb}

\newtheorem{thm}{Theorem}[section]
\newtheorem{cor}[thm]{Corollary}
\newtheorem{example}[thm]{Example}
\newtheorem{lem}[thm]{Lemma}
\newtheorem{prop}[thm]{Proposition}
\theoremstyle{definition}
\newtheorem{defn}[thm]{Definition}
\theoremstyle{remark}
\newtheorem{rem}[thm]{Remark}
\numberwithin{equation}{section}
\title{On the structure of the  ring of   integer-valued entire functions}
\author{\small Bechir Amri \\
\small bechiramri@69gmail.com, \\
\small University of Carthage, Department of Mathematics, Faculty of Sciences of Bizerte, Tunisia.  }
\date{}
\begin{document}
\maketitle
  \begin{abstract}
The present paper studies structure of the ring of integer-valued entire functions.  We characterize certain classes   of  prime and maximal ideals  and investigate some of their properties.
\end{abstract}
\section{Introduction}
The study of  algebraic and topological structures of rings of functions has been a subject of interest for many years.  In particular,   rings of   continuous functions on a topological spaces  and  of  holomorphic functions  have fairly rich structures,   especially in   characterization of prime and maximal ideals. A pioneer papers in this  subject are \cite{HM,EH,LG}.
\par This paper study  the  ring $\mathfrak{A}$ of    integer-valued entire functions,  that  is  the set of
 entire function in the complex plane and  takes values in $\mathbb{Z}$ at all points in $\mathbb{Z}$.  It has  been inspired and motivated by the study of the  ring of entire functions,  presented in the papers \cite{OH, MH}. Although they have similar properties,  the structure  of  ring $\mathfrak{A}$   is   quite  different: its not a Bezout ring and has  more classes of prime and  maximal  ideals.  However,  the ring $\mathfrak{U}$ is not similar in structure to that of  integer-valued polynomials, for example an integer-valued entire function is not necessarily continuous for p-adic metric, an interesting argument to classify maximal ideals. The ideals of $\mathfrak{A}$ are herein classified as in \cite{EH},    free and fixed,  according to the intersection of the sets of irreducible devisors of its elements is   EmptySet or not. We first  establish  necessary and sufficiently    condition  for finitely generated ideal to be principal. A such condition will be frequently invoked when studying the problem of identifying prime and maximal ideals.
\par
  \section{ Elementary properties of $\mathfrak{A}$  }
 Let us start with the following  lemma,
\begin{lem}\label{l1}
  For a given sequence $(a_k)_k$ of complex numbers, there exists an entire function $f$ such that for all $k\in \mathbb{Z}$,
 $ f(k)=a_k$.
\end{lem}
  This can seen  by virtue of Weierstrass  theorem  or directly  by taking the following function
   $$f(z)=\sum_{k\in \mathbb{Z}}a_k\left(\frac{ \sin(\pi (z-\pi))}{\pi(z-k)}\right)^{|k|+[|a_k|]}=\sum_{k\in \mathbb{Z}}u_k(z)$$
where $[x]$ denotes  the integer part of a real number $x$. The series is normally convergent in every compact set. In fact,  for   $M>0$ and  $|z|\leq M$
we have $|\sin\pi z|\leq e^{\pi M}$   and    $\pi|z-k|\geq 2e^{\pi M}$,  for  $|k|$ big enough.  Then
$$|u_k(z)|\leq \frac{|a_k|}{2^{|k|+[|a_k|]}}\leq \frac{1}{2^{|k|-1}}.$$
\begin{cor}\label{cor2.2}
For a given sequence $(a_k)_k$ of nonzero complex numbers, there exists a non-vanishing  entire function  such that
 $ f(k)=a_k$, for all $k\in \mathbb{Z}$.
\end{cor}
\begin{proof}
  As the exponential function  is surjective we let  $ a_k= e^{z_k}$,  $k\in \mathbb{Z}$. From the above lemma there exists   an entire function  $f$ such that  $ f(k)=z_k$, for all $k\in \mathbb{Z}$. Hence    a  desired function is given by $g(z)=e^{f(z)}$
\end{proof}
In what follows we denote by $\mathfrak{ A}$  the ring of  integer-valued entire functions,  that  is  the set of
entire  functions $f$  which  satisfy $f(\mathbb{Z})\subset  \mathbb{Z}$.
   From the isolated zeros  property   the  ring  $\mathfrak{ A}$ is an integral domain. However,   $f\in \mathfrak{ A}$ is invertible   if and only if   it is non-vanishing function with  $f(k)\in \{-1,1\}$, $k\in \mathbb{Z}$.
  \begin{prop}\label{p1}
The irreducible elements of $\mathfrak{ A}$ are functions $f$ that satisfy one of the following  conditions:\\
(i) $f$ does not vanish and there exists only one $n\in \mathbb{Z} $ such that $f(n)$ is a prime  integer  and $f(k)\in\{1,-1\}$ for all $k\in \mathbb{Z}$, $k\neq n$.\\
(ii) $f$   vanishes only  at one complex number $a\in \mathbb{C}\setminus\mathbb{Z}$ which is a simple zero and $f(\mathbb{Z})\subset\{1,-1\}$
\end{prop}
\begin{proof}
It is not hard to verify  that a function of $\mathfrak{ A}$  that  satisfying i) or ii) is irreducible. Conversely,
   let  $f$  be an irreducible element  of $\mathfrak{ A}$. Then for all $n\in \mathbb{Z}$,  $f(n)$  is irreducible or invertible in $\mathbb{Z}$. In fact, suppose that for some $n\in \mathbb{Z}$,
  $f(n)= rs $ with  $r,s \in \mathbb{Z}$,  $r\neq 0$ and $s\notin\{-1,1\}$.   Let $g\in \mathfrak{ A}$ be  a none-vanishing function such that $g(n)=r$ and $g(k)=1$, for all $k\in \mathbb{Z}$, $k\neq n$.
 Then  $g$ divides $f$ and $f/g$ is  not invertible since $f/g(n)=s$,   this   contradicts  the fact that $f$ is irreducible.
 Furthermore, we exclude the case  where   there are two distinct integers  $n_1, n_2  \in \mathbb{Z}$  such that
  $f(n_1)$ and $f(n_2)$ are irreducible  integers,  since   a none-vanishing function  $g\in \mathfrak{ A}$   such that $g(n_1)=f(n_1)$ and $g(k)=1$, for all $k\in \mathbb{Z}$, $k\neq n_1$ divides $f$. Now if   $f(\mathbb{Z})\subset \{-1,1\} $  then $f$ vanishes on $\mathbb{C}\setminus\mathbb{Z}$ since otherwise  $f$ is   invertible. Suppose that $f(a)=f(b )=0$  for  two distinct complex numbers   $a,b \in \mathbb{C}\setminus\mathbb{Z}$. Taking  a none vanishing  $h\in\mathfrak{ A}$ such that $h(k)=1/(k-a)$ for all $k\in \mathbb{Z}$, then the function $g(z)= (z-a)h(z)$ divides $f$. Also  $g$ divides    $f$  in the case  $f(a)=f'(a)=0$.   Now, by virtue of the above discussions the function $f$ must satisfy    condition    i) or ii).
  \end{proof}
The following corollary is an immediate consequence.
\begin{cor}
The prime elements of  $\mathfrak{ A}$ are exactly the irreducible elements.
\end{cor}
\begin{prop}
 $\mathfrak{ A}$   is  not a factorial ring.
\end{prop}
\begin{proof}
The function $f(z)=z$ has an infinite number of prime divisors,  since   for  $n\in \mathbb{Z}$ and $p$ a prime divisor of $n$ then the none vanishing  function $g\in  \mathfrak{ A}$  such that $g(n)=p$ and $g(k)=1$ for $k\neq n$ is a prime divisor of $f$.
\end{proof}
 \begin{prop}
  $\mathfrak{ A}$ is not a   Noetherian ring.
\end{prop}

\begin{proof}
  Consider the the sequence $(H_n)_n$ of $\mathfrak{ A}$,
  $$H_n(z)=\prod_{k=n}^{+\infty}\left(1-\frac{\sin^2(\pi z)}{k^2}\right).$$
  The ascending chains of ideals $(H_n)_n$ is not stable.
\end{proof}
 For $f$ is an entire function, we let $Z(f)$ the algebraic set of its zeros (  each zero occurring a number of times equal
to its multiplicity )
 \begin{lem}\label{l2}
   If $f,g\in \mathfrak{A} $ such that $Z(f)\cap Z(g)=\emptyset $ and $gcd(f(n),g(n))=1$   for all $n\in \mathbb{Z}$, then there  are $u,v\in \mathfrak{A}$ such that
   $uf+vg=1$.
 \end{lem}
 \begin{proof}
We know from Theorem 9 of \cite{OH}
that  there are  entire functions   $u_1$ and $v_1$ such that  $u_1f+v_1g=1$.  On the other hand for each $n\in \mathbb{Z}$ there exist $a_n , b_n \in \mathbb{Z}$ such that $a_n f(n)+b_ng(n)=1$.  We let $k$ be an entire function such that
   $$k(n)=\frac{a_n-u_1(n)}{g(n)}  \;\; \text{if $g(n)\neq0$ and } k(n)=\frac{v_1(n)-b_n}{f(n)}\;\;\text{if $g(n)=0$. }$$
Noting here that when both $g(n)$ and $f(n)$ are non-zero we have
$$\frac{a_n-u_1(n)}{g(n)}=\frac{v_1(n)-b_n}{f(n)} $$
Thus putting  $u=u_1+kg$ $v= v_1-kf$.  Clearly $uf+vg=1$ and  $u, v\in \mathfrak{A}$, since $u(n)=a_n$ and $v(n)=b_n$ for all  $n\in \mathbb{Z}$.
 \end{proof}
 \begin{lem}\label{l3}
   If $f,g\in \mathfrak{A} $ such that $Z(f)\cap Z(g)\cap \mathbb{Z}=\emptyset$ then  the ideal  $(f,g)=f\mathfrak{A}+g\mathfrak{A}$ is principal.
 \end{lem}
 \begin{proof}
   Let $d$ be an entire function such that $Z(d)=Z(f)\cap Z(g)$. We assume that $d(n)=gcd(f(n),g(n))$, for all $n\in \mathbb{Z}$ ( by virtue of corollary \ref{cor2.2} one can multiply $d$ by a none vanishing entire function $h$ such that $h(n)= d(n)^{-1}gcd(f(n),g(n)),\;  n\in \mathbb{Z} )$. Thus $f/d$ and $g/d$ are elements of $\mathfrak{A}$ that  satisfy
   the hypothesises  of Lemma \ref{l2} and from  which  it follows  that $f\mathfrak{A}+g\mathfrak{A}=d\mathfrak{A}$.
 \end{proof}
\begin{thm}\label{th1}
    Let $f_1,f_2,...f_r\in \mathfrak{A}$. The  finitely generated ideal   $ I= (f_1,f_2,..,f_r)$ is principal if and only if for each  $n\in \mathbb{Z}$ there exists $j_n\in \{1,2,...,r\}$ such that $(f_j/f_{j_n})(n)\in \mathbb{Q}$ for all $j= 1,2,...,r.$
\end{thm}
\begin{proof}
Let us  suppose first  that  $I$ is   principal ideal and put $ I=d\mathfrak{A}$ for some $d\in \mathfrak{A}$.  Clearly  $Z(d)=\cap_{j=1,2,...r} Z(f_j) $ and for $n\in \mathbb{Z}$ there exists $j_n$ such that $o(d,n)=o(f_{j_n}, n)$ (   multiplicity of $n$ ). As $(f_{j_n}/d)(n)\in \mathbb{Z}\setminus\{0\}$ and $(f_j/d)(n)\in \mathbb{Z}$ for all $j= 1,2,...,r$  then  $(fj/f_{j_n})(n)\in \mathbb{Q}$.
Conversely,  suppose first that $\cap_{j=1,2,...,r}Z(f_j)=\emptyset$.  Then for all $n\in \mathbb{Z}$ the integers   $f_1(n),\;f_2(n),...,f_r(n)$ are not all zero, and hence   one can choose   $a_{1,n},a_{2,n},...,a_{r,n}\in \mathbb{Z}$ such that
 $$a_{1,n}f_1(n)+a_{2,n}f_2(n)+...+a_{r,n}f_r(n)\neq 0.$$
Taking  $h_j\in \mathfrak{A}$, $j= 1,2,...,r$  such that $h_j(n)=a_{j,n}$, for all $n\in\mathbb{Z}$ and let  $h=h_1f_1+h_2f_2+...+,h_rf_r$. Now as
$ I= (h,f_1,f_2,...,f_r)$
 then we can  apply successively Lemma \ref{l2} to conclude that I is principal.
\par  Suppose now $\cap_{j=1,2,...,r}Z(f_j)\neq \emptyset$ and let  $g$ be an entire function such that $Z(g)=\cap_{j=1,2,...,r}Z(f_j)$. By hypothesis   we set
$$ \frac{f_j}{f_{j_n}}(n)=\frac{p_{j,n}}{q_{j,n}},\qquad\text{with}\qquad gcd( p_{j,n},q_{j,n})=1,  \quad  n\in \mathbb{Z},\; j=1,2,...,r.$$
Let  $q_n$ be the least common multiple of all   $q_{j,n }$, clearly this is a divisor of $f_{j_n}(n)$. Choosing a none vanishing entire function $\alpha$ such that
$$\alpha(n)=\frac{f_{j_n}}{q_n g}(n), \qquad n\in \mathbb{Z}$$
 and put $D= \alpha g$. Noting here that $(f_{j_n} / g)(n)\neq 0$,  since  otherwise  we would  have    $(f_j /g) (n)=0$  for all $j=1,2,...,r$  which cannot occur  because $Z(g)=\cap_{j=1,2,...,r}Z(f_j)$. Thus we have   $D\in \mathfrak{A}$ and divides all  $f_j$, $j=1,2,...,r$. As Now $\cap_{j=1,2,...,r}Z(f_j/D)=\emptyset$  then the ideal
$(f_1/D,f_2/D,...,f_r/D)$ is   principal  and it is the same for the ideal  $I$.
\end{proof}
 Taking the following example,  $f(z)=z$ and $g(z)= \sin(\pi z) $. Since $(g/f)(0)=\pi$, then the ideal $(f,g)$ is not   principal
 \begin{cor}
  An  invertible ideal of $\mathfrak{A}$ is  principal.
\end{cor}
\begin{proof}
  Suppose that $I$ is an invertible ideal of $\mathfrak{A}$, then  it's inverse is the $\mathfrak{A}$-module  given by $I^{-1}=\{h\in \mathfrak{K}, hI\subset\mathfrak{A}\}$, where
$\mathfrak{K}$ is the fraction field of the  ring $\mathfrak{A}$.
It follows  that  $I$  and $I^{-1}$  are  finitely generated,   let  $I=( f_1,f_2,...,f_s)$ and $I^{-1}=(h_1,h_2,...,h_s)$ with
$\sum_{j=1}^sh_jf_j=1$. Thus  for $n\in \mathbb{Z} $ there exists $j_n$ such that $(h_{j_n}f_{j_n})(n)\neq0$, hence
$$\frac{f_j}{f_{j_n}}(n)=\frac{(h_{j_n}f_j)(n)}{(h_{j_n}f_{j_n})(n)}\in \mathbb{Q},\qquad  j=1,2,...,s$$
and $I$ is principal by Theorem \ref{th1}.
\end{proof}
\section{Maximal and prime ideals }
There are two remarkable examples of maximal ideals. The First type   is  given by ideals
$M_{a}=\{f\in \mathfrak{ A},\; f(a)=0\}$,   for  $a\in \mathbb{C}\setminus\mathbb{Z}$ and the second type  is  given by ideals
    $M_{n,p}=\{f\in \mathfrak{ A},\; f(n) \equiv0 (p)\}$  for  $n\in\mathbb{Z} $ and a  prime number $p$.
\begin{defn}
Let $ \Pi(f)$ be  the set of the   irreducible  divisors of    $f\in \mathfrak{A}$.  For $\delta \in  \Pi(f)$,   we  define
$$o(\delta, f)=\max\{k ; \; \delta^k divides f  \}, \qquad o(\delta, 0)=\infty.$$
 and  $$m(f)=\sup \{o(\delta, f), \delta\in \Pi(f)\}$$
\end{defn}
Observe that when $m(f)=1$ then all zeros of $f$ are simple  and  $f(n)$ is a free integer square for all $n\in \mathbb{Z}$. In addition, if
$\Pi(f)\cap \Pi(g)= \emptyset$ then $Z(f)\cap Z(g)=\emptyset$ and $gcd( f(n),g(n))=1$ for all $n\in \mathbb{Z}$.

\begin{lem}
   Each (proper) maximal ideal of $\mathfrak{A}$  contains an element $f$  with $m(f)=1$.
\end{lem}
\begin{proof}
  Let $M$ be a  maximal   ideal. Let see first that $M$ contains a function not vanishing  on $\mathbb{Z}$. indeed,  its  obvious that  $2$ or $3$ does not belongs to $M$. If   $2\notin M$ (similarly if $3\notin M $ )
  and as $A/M$ is a field then  there is
  $ f \in M$ and $g\in \mathfrak{ A} $  such that $ 2g+f=1$  which asserts   that the function   $f$ doest not vanish on $\mathbb{Z}$. Now  for an each $f$ we  let $h\in \mathfrak{ A} $
 a  none vanishing function     such that
$$ h(n)= \prod_{p/f(n)} p, \qquad n\in \mathbb{Z} $$
 where the product is over  prime integers divisors of $f(n)$ .
If $h\in M$  then $h$ is a desired function and the lemma is proved.   Otherwise  there exist $u\in \mathfrak{ A}$ and $v\in M $   such that $uh+v=1$. This    implies that  for  $n\in \mathbb{Z}$,
$gcd (h(n),  v(n))=1$ and
 consequently   $gcd( f(n), v(n))=1$. Thus from Lemma \ref{l1}  there exist $ \alpha ,\; \beta\in \mathfrak{ A}$ such that
 $d=\alpha f+\beta v \in M$ with  $ d(n)=1$ for all $n\in \mathbb{Z}$. Now  observe  that $Z(d)\neq \emptyset$, since otherwise $d$ is invertible. Then by Weierstrass theorem  one can construct  entire function $  \mathfrak{d}$ that has the same zeros as $d$ with  multiplicities equal to one  and such that  $\mathfrak{d}(n)=1$ for all $n\in \mathbb{Z}$. If $\mathfrak{d}\notin  M$ then there exist  $a\in \mathfrak{A} $ and $b\in M$ such that $a\mathfrak{d}+b=1$. Hence     $  Z(d) \cap Z(b)=\emptyset$ and   from Lemma \ref{l2}  there are  $r,s\in \mathfrak{A}$ such that
 $rd+sb=1\in M$, this yields a contradiction since  $M\neq \mathfrak{A}$. This  achieves the proof of the lemma.
\end{proof}
\begin{thm}\label{th2}
  A prime ideal $P$ of $\mathfrak{A}$ is  maximal if and only if   $P$  contains   an element $f$ with $m(f)=1$.
\end{thm}
\begin{proof}
 Let $M$  be a maximal ideal containing $P$ and
  let $f\in P$  with $m(f)=1$. We will prove that $M=P$.
 In fact, let $h\in M$,   there exists $d\in \mathfrak{A}$ such that $(f,h)=(d)$.  Then one can write $f=g d$ with
 $Z(g)\cap Z(d)=\emptyset$ and $ gcd(g(n), d(n))=1$, for all $n \in \mathbb{Z}$. In view of Lemma \ref{l1} the function  $g$ can not belongs to $M$ and  thus $d\in P $, since    $P$ is a  prime ideal. Therefore  $h\in P$.
\end{proof}
\begin{cor}\label{cgd}
   All finite family  of functions  of a  maximal ideal  $M$ of $\mathfrak{A}$ have a common divisor in $M$.
\end{cor}
\begin{proof}
Let $f_1,f_2,...,f_r$ be a functions in $M$ and let $f\in M$ with $m(f)=1$ ( which does not vanish on $\mathbb{Z}$). By  Theorem \ref{th1}
there is exists $d\in   \mathfrak{A}$ such that  ideal  $(f,f_1,f_2,...,f_r)=(d)$. Then we must have $d\in M $ and  it is    a common divisor of
  $f_1,f_2,...,f_r$.
\end{proof}
 \begin{thm}\label{th3}
   Let  $P$ be a  prime ideal of $\mathfrak{A}$ such that $\cap_{f\in P} \Pi(f)\neq \emptyset$.
  \\
   (i) If  $\cap_{f\in P} Z(f)\neq \emptyset$.      then there exists $a \in \mathbb{C}$ such that  $P=\{f\in \mathfrak{A},\; f(a)=0\}$. \\
   (ii) If $\cap_{f\in P} Z(f) = \emptyset$  then there exist $n\in \mathbb{Z}$ and a prime integer $p$ such that 
$P=\{f\in \mathfrak{A},\; f(n)\equiv 0(p)\}$.
   \end{thm}
\begin{proof}
  (i) We let $D=\cap_{f\in P} Z(f)$.
  We  suppose first  that  $D\cap\mathbb{Z}=\emptyset$. Let $a\in D$ and $h$
  be an none vanishing entire function such that
  $h(n)=(n-a)^{-1}$ for all $n\in\mathbb{Z}$.  For  $f\in P$  we let       $g=((z-a)h)^{k_0}$  where $k_0= o(f,a)$. Clearly $g\in\mathfrak{A}$ and divides $f$. Moreover, since $P$ is a prime ideal,  $ g $ belongs to $P$. Therefore $d=(z-a) h$ is an element of $P$ which irreducible and $P$ is the maximal ideal $\{\mathfrak{f}\in \mathfrak{A},\; \mathfrak{f}(a)=0\}$.
  Suppose  now $D\cap\mathbb{Z}=\emptyset$. Let $n\in D\cap\mathbb{Z}$ and $f\in P$  written as   $f=(z-n)^r\ell$, with $\ell(n)\neq 0$. Given $ \alpha$ a none vanishing entire function such that
  $\alpha (k)=(k-n)^r$, for $k\in \mathbb{Z}$, $k\neq n$ and $\alpha(n)=1/\ell(n)$.  Since ze have   $\alpha\ell\notin P$ and $(z-n)^{r+1}\alpha\ell\in P$, then we conclude that  $P$  contains the polynomial $z-n$. Now for $b\in \mathfrak{A}$  with   $b(n)=0$, we write $b^2=(z-n)a$. Since we have $a(k)\in \mathbb{Q}$ for all $k\in \mathbb{Z}$ then one can find   a none vanishing entire function $\beta $ such that $\beta a\in \mathfrak{A}$. Thus we have  $\beta b^2\in P$ and so $b\in P$. This proves that $P=\{\mathfrak{f}\in \mathfrak{A},\; \mathfrak{f}(n)=0\}$.\\
ii) In this case  the  functions  in $P$ have a common  irreducible divisor of type $\delta_{n,p}$. It's not hard to see that $\delta_{n,p}\in P$ and then $P=\{f\in \mathfrak{A},\; f(n)\equiv 0(p)\}$.
\end{proof}
 \begin{cor}
   Let $n\in \mathbb{Z}$ and $p$ be a prime integer. The maximal ideal $M_{n,p}=\{f\in \mathfrak{A},\; f(n)\equiv 0(mod \;p)\}$ contains one and only one proper prime ideal, namely  $P_{n}= \{f\in \mathfrak{A},\; f(n)=0 \}$. Moreover, the   maximal ideals containing  $P_n$ are all in the  form of $M_{n,p}$.
 \end{cor}
\begin{proof}
  Let $P$ be a proper prime ideal included in $M_{n,p}$.  Suppose that there   is  $f\in P$   with $f(n)\neq 0$ and write  $f(n)=p^r q$ with $gcd(p, q)=1$. Let $h$ be an irreducible element of $\mathfrak{A}$ such that $ h(n)=p$ and $h(k)=1$ for $k\in \mathbb{Z}$, $k\neq n$. Then    $f=h^r g$ with
$gcd( g(n), p)=1$ and $g\notin M$. As $P$ is a prime ideal  it follows that   $h \in  P$ and  hence $P=M_{n,p}$, contradiction.
We conclude that $P= \{f\in \mathfrak{A},\; f(n)=0 \}$ by using  Theorem \ref{th3}.
 \par Let now $N$ be a (proper)  maximal ideal that  contains $P_n$ for some $n\in \mathbb{Z}$. Let $f\in N$ and $f\notin P_n$. We will show that $f(n)\neq \pm1$. In fact,  suppose   that $f(n)=\pm 1$ and  choose
a none vanishing entire function $\alpha $  with  $\alpha(k)=(k-n)^{-1}$,  for all $k\in \mathbb{Z}$, $k\neq n$  and $\alpha(n)=1$.  Put  $h=(z-n)\alpha\in P_n$.
We have $Z(h)\cap Z(f)=\emptyset$ and $gcd( h(k), f(k))=1$, for all $k\in \mathbb{Z}$ . But  this leads to
 $N= \mathfrak{A}$,  thanks to Lemma \ref{l2}, so we must have $f(n)\neq\pm 1$. Now  decompose into prime factors $f(n)=\pm p_1^{r_1}p_2^{r_2}...p_\ell^{r_\ell}$
 and  let $h_1,h_2...,h_\ell$ be  irreducible elements of $\mathfrak{A}$ such that $h_j(n)=p_j$, $j=1,2,...,\ell$. Thus one can write
$f=h_1^{r_1}h_2^{r_2}...h_\ell^{r_\ell}g $ with $g(n)=1$. As $g\notin N$, the    $N$ contains one of $h_j$ and    thus
$N=N_{n,p_j}\{f\in \mathfrak{A},\; f(n)\equiv0( p_j)\}$, again thanks to Lemma \ref{l1}.
\end{proof}
\begin{thm}\label{th4}
   Let $P$ be a prime ideal of $\mathfrak{A}$   containing   a function   which  does not vanish on $\mathbb{Z}$ and let $f\in \mathfrak{A}$
   with $m(f)=1$.  Then the ideal  $M=(P,f)$ is   maximal.
\end{thm}
\begin{proof}
 It's enough to show that $M$ is a prime ideal.  Suppose  first    $f=uv$ for some $u,v\in \mathfrak{A}$ and show that  $u\in M$ or $v\in M$.
 In fact, let  $g\in P$ that does not vanish on $\mathbb{Z}$.
  Taking $ h \in  \mathfrak{A} $ as follows:   the set of zeros of $h$ is the intersection of those of $u$ and $g$ 
  and each zero has the same  order  as zero of g and further for $n\in \mathbb{Z}$,
$$h(n)=\left\{
    \begin{array}{ll}
     \prod_{p /u(n)}p^{v_p(g(n))}, & \hbox{ if $u(n)\neq 1$;} \\
     1, & \hbox{ if $u(n)=1$.}
    \end{array}
  \right.
$$
where the product is taken on the  prime integers divisors of $u(n)$ and $v_p(g(n))=\max \{e;\; p^e \;divides \; g(n)\}$.  Now clearly $g$   is   written  as  $g=hr$  with    $(h,v)=(r,u)=\mathfrak{A} $. Hence  $u=\alpha f+\beta h u$ and $v=\theta f+\lambda rv$, for some $\alpha,\beta,\theta,\lambda\in \mathfrak{A}$. As $P$ is a prime ideal of $\mathfrak{A}$ thus we have that   $h\in P$ or $r\in P$ and  consequently     $u\in M$ or $v\in M$.
\par Now suppose that we have $f_1f_2\in M$  with
\begin{equation}\label{e2}
  f_1f_2= af+p
\end{equation}
for some $a\in  \mathfrak{A}$ and  $p\in P$.  Suppose  further that $f_1\notin M$ and let
     $f_1=d f_1^*$ and $ f=d f^*$ with  $(f_1^*,f^*)= \mathfrak{A}$. So, from  the above   $f^*$ must be in $M$.   Now  in view of (\ref{e2}),  $d(f_1^* f_2-af^*)=p$,  then  $f_1^* f_2-af^*\in P$ and $f_1^* f_2\in M$.  However,  if we write
1= $\mu f_1^*+\nu f^*$, some $\mu,\nu\in \mathfrak{A}$,  then we obtain that  $f_2=\mu f_2f_1^*+\nu f_2f^*\in M$. This achieves the proof of the theorem.
\end{proof}
\begin{cor}
 Let $P$ be a prime  ideal of $\mathfrak{A}$   containing  a function  that   does not vanish on $\mathbb{Z}$. Then $P$ is included in one and only one maximal ideal of $\mathfrak{A}$.
\end{cor}
\begin{proof}
Suppose that we have $M_1$ and $M_2$ two proper maximal ideals of $ \mathfrak{A}$ containing $P$. Let $f_1\in M_1$ and $f_2\in M_2$ with
$m(f_1)=m(f_2)=1$. Then from Theorem \ref{th4}, $M_1=(P,f_1)$ and  $M_2=(P,f_2)$. Now  let $(d)=(f_1,f_2)$ and write
$f_1=df_1^*$, $f_2=df_2^*$ with $ ( f_1^*, f_2^*)=\mathfrak{A}$. Suppose that $d\notin M_1$ and $d\notin M_2$ then we have $M_1=(P,f_1^*)$ and  $M_2=(P,f_2^*)$. Now for a given $g\in P$ that  does not vanish on $\mathbb{Z}$  we can  split $g=g_1g_2$ with $(g_1,f_1^*)=(g_2,f_2^*)=\mathfrak{A}$.
 But  this yields  that    $M_1=\mathfrak{A}$ or $ M_2=\mathfrak{A}$,  since  we have $g_1\in P$ or $g_2\in P$ which   contradicts the fact that  $M_1$ and $M_2$ are proper ideals of $\mathfrak{A}$. Therefore we have  $d\in M_1$ or  $d\in M_2$  which imply that $M_1=M_2$.
\end{proof}
\begin{thm}
  Let $P$ be a prime  ideal of $\mathfrak{A}$   containing  a function  that   does not vanish on $\mathbb{Z}$. The residue class ring  $\mathfrak{A}/P$  is a valuation ring whose the unique maximal  ideal $M/P$ is principal.
\end{thm}
\begin{proof}
We suppose that $P$  is not maximal. Let us  denote $\bar{f}$ the class modulo $P$ of $f\in \mathfrak{A}$. Noting that   $\mathfrak{A}/P$  is a valuation ring if and only if for all
$f,g\in \mathfrak{A}$ we have $\bar{f}$ divides $\bar{g}$ or $\bar{g}$ divides  $\bar{f}$.
 Let $M$ be the maximal ideal containing $P$. As $M$ is the only maximal ideal containing $P$ then  for all  $f\notin M$    the ideal $(P,f)=\mathfrak{A} $ and there exists $h\in \mathfrak{A} $ such that $fh-1\in P$, implies that   $\bar{f}$ is  a unit in the ring $\mathfrak{A}/P$.
Let now  $f, g\in M\setminus P$ and $h\in P$  that does not vanish on $\mathbb{Z}$. Then if we let  $(d)= (f,g,h) $,  we  must have  that  $d\notin P$ and  $h/d\in P$, which imply that  $f/d$ or  $g/d$ does not belong to $M$ and thus is a unit modulo $P$. Hence we conclude that   $\bar{f}$ divides $\bar{g}$ or $\bar{g}$ divides  $\bar{f}$.  Further, if we  take  $f\in M\setminus M^2$ then it is  clearly that $M/P$ is generated by $(\bar{f})$.
\end{proof}
\begin{thm}
  Let $\mathfrak{E}$ be the ring of entire functions.
\\(i) The mapping $P\longrightarrow P\cap \mathfrak{A}$  is injective that does not surjective from the set of prime ideals of $\mathfrak{E}$
into  the set of prime ideals of $\mathfrak{A}$
\\(ii) If $M$  is  a maximal ideal of $\mathfrak{E}$ then $M\cap \mathfrak{A}$ is a maximal ideal of
$\mathfrak{A}$ if and only if $M$ contains  a function   which does not vanish on $\mathbb{Z}$.
\end{thm}
\begin{proof}
(i) Suppose  $P_1$ and $P_2$ are distinct prime ideals of $\mathfrak{E}$ and there is $u\in P_1$  and $ u\notin P_2$. Let $v\in \mathfrak{E}$ be  a none vanishing entire function such that $uv\in \mathfrak{A}$. Then $uv\in P_1\cap \mathfrak{A}$ and $uv \notin P_2\cap \mathfrak{A}$ and so  $ P_1\cap \mathfrak{A}\neq  P_2\cap \mathfrak{A}$.  For the surjectivity we note  first that
if $P_1\cap \mathfrak{A}\subset P_2\cap \mathfrak{A}$ then $P_1\subset P_2$. In fact, let $g\in P_1$ and $h\in \mathfrak{E}$ a none vanishing function
such that $hg\in \mathfrak{A}$. Then we have that $hg\in P_2\cap \mathfrak{A} $ and $g\in P_2$, since $h\notin P_2$.
Consider now the prime ideal of $\mathfrak{A}$, $M=\{f\in \mathfrak{A},\; f(0)\equiv 0(2)\}$ and suppose that $M= P\cap \mathfrak{A}$ where $P$ is a prime ideal of $\mathfrak{E}$. However $I=\{f\in \mathfrak{E},\; f(0)=0\}$ is a maximal ideal of $\mathfrak{E}$ and $I\cap \mathfrak{A}\subset P\cap\mathfrak{A}$ which implies that $I\subset P$ and  then $I=P$ which cannot hold.
\\(ii) Let $M$ be a  (proper) maximal ideal of $\mathfrak{E}$. We will  affirm   that  $\mathfrak{A}/\mathfrak{A}\cap M$ is a field.
  Let $h\in M$ assumed   not   vanish on $\mathbb{Z}$ and let $f\in \mathfrak{A}$, $f\notin M$. There exists $g\in\mathfrak{E}$ such that
$gf-1\in M$. Let $\kappa \in \mathfrak{E}$ such that
$$\kappa(n)=\frac{g(n)}{h(n)}, \quad n\in \mathbb{Z}.$$
Putting $L=\kappa h\in M$ and $H=g-L\in \mathfrak{A}$.
Thus we have $Hf-1=gf-1-Lf \in M$ and so $f$ is invertible in $\mathfrak{A}/\mathfrak{A}\cap M$.  To prove the converse, assume that $M\cap \mathfrak{A}$ is a maximal ideal of
$\mathfrak{A}$. Since the  constant function $ 2$ does not belong to   $M\cap \mathfrak{A}$, then there is $d\in \mathfrak{A}$ such that
$2d-1\in  M\cap \mathfrak{A}$ and this function does not vanish on $\mathbb{Z}$.
\end{proof}
\begin{thm}
  The ring $\mathfrak{A}$ is  integrally closed.
\end{thm}
\begin{proof}
Let us observe that   the rings $\mathfrak{A}$ and $\mathfrak{E}$ have the same fractions field, namely the field of the meromorphic functions,
since   for $f,g\in \mathfrak{E}$ we can write
$$\frac{f}{g}(z)=\frac{f(z)\sin(\pi z)}{g(z)\sin(\pi z)}.$$
Now, $\mathfrak{A}$ is integrally closed  can be easily deduced   from  fact that  $\mathfrak{E}$  and $\mathbb{Z}$ are   integrally closed ring (  gcd-domains ).
\end{proof}
\begin{thm}
 The ring $\mathfrak{A}$  is nearly analytic,  namely that\\
i) if $M $ is a maximal ideal of   $\mathfrak{A}$ then its powers $M^n$
for $n = 1, 2, . . . $ are distinct,\\
ii) $P^*=\cap_{n=1}^\infty M^n$ is a prime ideal of $\mathfrak{A}$  and is the largest non-maximal prime ideal  contained in M.
\end{thm}
 \begin{proof}
i) We proceed as in \cite{MH}. First we note  that if $g\in M^n$  for some positive integer $n$ then  $g=ud^n$,  where  $u\in  \mathfrak{A}$ and $d\in M$.  This  can be  follow by recurrence on $n$. In fact, let
 $g\in  M^{n+1}=MM^n$, which can be written as   a finite sum of the form  $ \sum_i a_ib_i^n$  with $a_i,b_i \in M$.  Choose a function  $h\in M $ that does not vanish on $\mathbb{Z}$ and   let $(d)=(a_1,a_2,,,b_1,b_2,...,h)$, it follows that   $d \in M $  and  hence   $g$ takes the form  $g=ud^{n+1}$.
 Now using this fact we deduce that  for   $f\in M$ with $m(f)=1$   we have   $f^n\in M^n\backslash M^{n+1}$, for all positive integer $n$.
\\ ii) Suppose we have $f_1f_2\in P^*$. If $f_1\notin M$ (or $f_2\notin M)$ then there are $\alpha\in \mathfrak{A}$ and $m\in M$ such that  $\alpha f_1+m=1$
 and one  can write  $f_2= f_1f_2+mf_2$,  from which   $f_2\in P^*$.
If both $f_1,f_2\in M\setminus P^*$ then there exist  a positive integers $r,s$  such that  $f_1\in M^r\setminus M^{r+1}$ and    $f_2\in M^s\setminus M^{s+1}$, and one write 
$f_1=u(f_1^*)^r$ and $f_2=v(f_2^*)^s$ where $u,v\in \mathfrak{A}\setminus M$ and $f_1^*,f_2^*\in M\setminus M^2$. On the other hand ,  in view  of the corollary \ref{cgd} if we  let $d\in M\setminus M^2$ be a common divisor of $f_1^*$ and $f_2^*$ then we must have
$f_1f_2= wd^{r+s}$ where  $w\in \mathfrak{A}\setminus M$. Now as $f_1f_2\in P^*$,  its written as  $f_1f_2= ab^{r+s+1}$  with $b\in M$.
Letting $\delta\in M$   be a common divisor of $d$ and $b$,    we have    $f_1f_2= \alpha \delta^{r+s}=\beta \delta^{r+s+1}$ with
 $\alpha\in \mathfrak{A}\setminus M$,   but this does not carry out since  $\mathfrak{A}$ is an  integral domain, 
and  thus    $f_1$ or $f_2$ must be in $P^*$.  Hence  we conclude that $P^*$ is a prime ideal.
\par In order to prove  the second part of the theorem,  we let  $P$ be a non-maximal prime ideal contained in $M$ and  $f\in M$ with $m(f)=1$.
Let $g\in P$ and $d\in M$ be a common divisor of $f$ and $g$, where we  write $f=df^*$ and $g=dg^*$. Since $m(d)=1$ then $d\notin P$ and  thus
$g^*\in P$.   Hence we conclude  that $ MP= P$ and    this fact  lead to the inclusion of  $P$ in $ P^*$.
\end{proof}
In the next we  let $M$ be a proper  maximal ideal of $\mathfrak{A}$. Following \cite{MH},  for  $f,g\in M$  we have
    $f \geq  g $  if  there is $e\in M$ such that
$$\Pi(e)\subset \Pi(f)\cap \Pi(g)
 \quad\text{and}\quad o( f,\delta)\geq o(g,\delta), \qquad \text{ for all } \; \delta\in \Pi(e).$$
  We denote  $f\gg g$ if $f\geq g^N$ \  all positive integers $N$.
\begin{rem}\label{rem}
  The function $e$ can be taken with $m(e)=1$. Indeed, let $h\in M$ with  $m(h)=1$ and $d\in M$ such that $(e,h)=(d)$. 
Then  $\Pi(d)\subset \Pi(e)$ and  $m(d)=1$. 
\end{rem}
\begin{lem}\label{fg}
  If $f,g\in M$ then $f\geq g$ or $g\geq f$.
\end{lem}
\begin{proof}
  Let $h\in M$ with $m(h)=1$ and $d\in M$ such that $(f,g,h)=(d)$. Clearly f $\{\delta \in \Pi(d);\; o(f,\delta)\geq o(g,\delta)\} =\emptyset$  then $g\geq f$. Otherwise, let $v\in \mathfrak{U}$ be a divisor of $d$ such that $\Pi(v)=\{\delta \in \Pi(d);\; o(f,\delta)\geq o(g,\delta)\} $ and
$o(v,\delta)=o(d,\delta)$ for all $\delta\in \Pi(v)$. If $v\in M$ then clearly  $f\geq g$. If $v\notin M$ then $d/v\in M $ and we must have
$g\geq f$.
\end{proof}
\begin{lem}\label{le}
If $f_1, f_2 , g\in M$ such that $f_1f_2\geq g^{2N}$ for some  positive integer $N$, then we have  $f_1\geq g^N$ or $f_2\geq g^N$.
\end{lem}
\begin{proof}
Let   $e\in M$ with $m(e)=1 $, such that $\Pi(e)\subset \Pi(f_1f_2)\cap \Pi(g)$
 and
$o( f_1f_2,\delta)\geq o(g^{2N},\delta)$,  for all $\delta\in \Pi(e).$
Assume that 
$A = \{\delta\in \Pi(e); \; o( f_1,\delta)  \geq  o(g^N,\delta)\} \neq \emptyset$
and let  $d \in \mathfrak{U}$ be a divisor of $e$, so that $\Pi(d)=A$.
  Therefore, when   $d\in M$  we have that $f_1\geq g^N$, otherwise we have $e/d\in M$ and
  $f_2\geq g^N$, noting here that $\Pi(e/d)=\{\delta\in \Pi(e); \; o( f_2,\delta)  >  o(g^N,\delta)\}$. This completes the proof  and the lemma
\end{proof}
\begin{thm}\label{pom}
  Let $\Omega$ a subset of $M$ and
$$P_\Omega=\{f\in M;\; f\gg g\; \text{for all} \; g\in \Omega\}.$$
Then  $P_\Omega$ is a prime ideal.
\end{thm}
\begin{proof}
We   prove first  that $P_g$, $g\in M$ is  an  ideal and its prime. Let $N$ be a positive integer.
For  $f_1,f_2 \in P_g$,   there exist  $e_1,e_2\in M$ such that we have the following,
$$\Pi(e_1)\subset \Pi(f_1)\cap \Pi(g),\;  \Pi(e_2)\subset \Pi(f_2)\cap \Pi(g),\;\;
 o( f_1,\delta_1)\geq o(g^N,\delta_1),$$
$$   o( f_2,\delta_2)\geq o(g^N,\delta_2)$$
  for all  $\delta_1\in \Pi(e_1)$ and  $\delta_2\in \Pi(e_2)$.
 Let $h\in M$ be a function that does not vanish on $\mathbb{Z}$ and let $d\in M$ such that $(d)=(e_1,e_2, h)$.  Clearly
$$\Pi(d)\subset \Pi(f_1)\cap \Pi(f_2)\cap\Pi(g)\subset \Pi(f_1+f_2) \cap\Pi(g)$$
and for $\delta\in \Pi(d),$
$$o( f_1+f_2,\delta)\geq \min(o( f_1,\delta),o( f_2,\delta))\geq o(g^N,\delta)$$
 Hence we have that $f_1+f_2 \in P_g$. Similarly one can see that for $f_1\in P_g$ and $f_2\in \mathfrak{A}$  the product $f_1f_2\in P_g.$  We thus  conclude that    $P_g$ is an ideal.
 \par Now,  we  show  that $P_g$ is prime.   Assume that   $f_1f_2\in P_g$, for $f_1, f_2\in  \mathfrak{A}$. In particular $f_1f_2\in M$ and then  $f_1$ or $f_2$ must be in $M$. If first both $f_1$  and $f_2$ belong to $M$ then by  Lemma \ref{le} one has
 $$f_1\geq g^N\quad \text{or}\quad f_2\geq g^N$$
for all positive integers $N$. As  at least one of the sets $\{N; \; f_1\geq g^N\}$ and $\{N; \; f_2\geq g^N\}$ has infinitely many positive integers $N$
then we must have $f_1\in P_g$ or $f_2\in P_g$. Otherwise,    suppose that    $f_1\notin M$.  Let  $N$ be a positive integer and  $e\in M$ with $ m(e)=1$ (see remark \ref{rem}) and   such that
$\Pi(e)\subset \Pi(f_1f_2)\cap \Pi(g)$ and
$$o( f_1f_2,\delta)\geq o(g^N,\delta),\quad \text{for all }\; \delta\in \Pi(e).$$
Putting  $e=e_1e_2$   with  $\Pi(e_1)\cap \Pi(f_1)=\emptyset$ and $\Pi(e_2) \subset \Pi(f_1)$. Clearly we cannot have
 $e_2\in M$, since in case  $M$ contains an element  $\omega$ so that  $\Pi(e_2)\cap \Pi(\omega)=\emptyset$.  Hence  $e_1\in M$ and  as
$\Pi(e_1)\subset \Pi(e)\subset \Pi(f_1f_2)\cap \Pi(g)$ and
$$o( f_2,\delta)=o( f_1f_2,\delta)\geq o(g^N,\delta),\quad \delta\in \Pi(e_1)$$
 then  $f_2\geq g^{N}$. Since  $N$ is arbitrary then $f_2\gg g$.
\par Now we come back  to the ideal  $P_\Omega=\cap_{g\in \Omega} P_g$ and prove that is prime.  If we let
 $f_1, f_2\in M\setminus P_\Omega$  then  there are
$g_1,g_2\in M$ such that $f_1\notin P_{g_1}$ and $f_2\notin P_{g_2}$. But in view of Lemma \ref{fg} we have   $P_{g_1} \subset P_{g_2}$ or $P_{g_2} \subset P_{g_1}$
 which implies that  $f_1.f_2\notin P_{g_1}$ or  $f_1.f_2\notin P_{g_2}$ and hence  $f_1.f_2\notin P_\Omega$. This completes the proof of Theorem \ref{pom}.
\begin{example}
  Let $p$ be a prime integer. Consider the ideal  of functions $f\in \mathfrak{A}$  such the set $\{n\in \mathbb{Z},\; gcd( f(n), p)=1\}$ is finite. 
 Let $M$ be a maximal ideal contains $I$. For a nonnegative integer $k$   taken   $f_k \in M$ such that 
$$f_k(n)= \left\{
            \begin{array}{ll}
              1, & \hbox{ if $|n|\leq k $;} \\
              p^{n^k}, & \hbox{ if $|n|> k $;.}
            \end{array}
          \right.
$$

Then we have $f_{k+1} \gg f_k$ and $P_{f_{k}} \varsubsetneq P_{f_{k+1}} $. The Ring   $\mathfrak{A}$  has infinite Krull dimension.
\end{example}

\end{proof}

\end{document}